\documentclass[12pt]{article}
\usepackage{amssymb,amsmath}
\usepackage{latexsym,bm}

\usepackage{amsfonts}
\usepackage{latexsym,amsmath,amssymb,amsfonts,epsfig,psfrag,url,graphics,ifpdf,multicol}
\usepackage{eepic,color,colordvi,amscd,amsthm}
\usepackage[section]{algorithm}
\usepackage{algpseudocode}
\usepackage[numbers,sort&compress]{natbib}
\usepackage{indentfirst,graphics,epsfig}
\usepackage{graphicx}
\usepackage{graphics}
\usepackage{tikz}

\newtheorem{theorem}{Theorem}
\newtheorem{lemma}[theorem]{Lemma}
\newtheorem{corollary}[theorem]{Corollary}
\newtheorem{problem}[theorem]{Problem}
 \theoremstyle{definition}
\theoremstyle{remark}

\def\V{\mathcal{V}}
\def\A{\mathcal{A}}
\newcounter{casenum}[theorem]

\newcounter{subcasenum}[theorem]

\newcounter{claimnum}[theorem]

\allowdisplaybreaks
\usepackage{indentfirst,subfig}

\begin{document}
\thispagestyle{plain}

\begin{center} {\Large On $k$-idempotent 0-1 matrices
}
\end{center}
\pagestyle{plain}
\begin{center}
{
  {\small   Zejun Huang$^{a}$, Huiqiu Lin$^{b,}$\footnote{Corresponding author.\\ Email: mathzejun@gmail.com (Huang); huiqiulin@126.com (Lin)}\\[3mm]
  { {\small $^{a}$College of Mathematics and Statistics, Shenzhen University, Shenzhen 518060, China }\\
  {\small $^{b}$Department of Mathematics, East China University of Science and Technology, Shanghai 200237,   China }}\\

}
}
\end{center}

\begin{center}

\begin{minipage}{140mm}
\begin{center}
{\bf Abstract}
\end{center}
{\small  Let  $k\ge 2$ be an integer. If a square 0-1 matrix $A$ satisfies $A^k=A$, then  $A$ is said to be   $k$-idempotent. In this paper, we give a characterization of  $k$-idempotent 0-1 matrices. We also determine the maximum number of nonzero entries in $k$-idempotent 0-1 matrices of a given  order as well as the $k$-idempotent 0-1 matrices attaining this maximum number.

 }

\end{minipage}
\end{center}

\section{Introduction}
 Nonnegative matrices play an important role in combinatorial matrix theory. Lots of interesting results on  combinatorial properties of powers of nonnegative matrices have been discovered. One of the most famous results on nonnegative matrices is the Perron-Frobenius Theorem, which states that the spectral radius of a square nonnegative matrix is one of its eigenvalue. A {\it primitive} matrix is a nonnegative matrix who has exactly one eigenvalue with modulus equal to its spectral radius. Frobenius found the connection between the primitivity and the powers of nonnegative matrices. He proved that a nonnegative matrix $A$ is primitive if and only if there is a positive integer $k$ such that $A^k$ is a positive matrix  (see \cite{XZ}, p.134). By this, given any primitive matrix $A$, there always exists a  smallest integer $k$ such that $A^k$ is positive, which  is called the {\it exponent} of $A$. Wielandt \cite{HW} proved that the exponent of an $n\times n$ primitive matrix is less than or equal to $(n-1)^2+1$. Dulmage, Mendelsohn Lewin, Vitek, Shao and Zhang \cite{DM, ML, LV,JS1,KZ} determined all the possible exponents of primitive matrices of order $n$.

Brualdi and Lewin \cite{BL} investigated the structure of powers of nonnegative matrices and  characterized the nonnegative matrices for which some power is   triangular.
Heap and Lynn \cite{HeL1,HeL2} studied {\it periods} and the {\it indices of convergence} of  nonnegative matrices, which are defined based on the zero patterns of powers of nonnegative matrices and have connections with   stochastic theory.

Among all nonnegative matrices, 0-1 matrices are very special, since they have close relationships with graph theory. Combinatorial properties of powers of 0-1 matrices can represent certain properties of digraphs, which have attracted lots of attentions. Huang, Lyu, Qiao, Wu and Zhan \cite{HL1,HL3,HZ1,HZ2} investigated 0-1 matrices whose powers are also 0-1 matrices. Huang and Lyu \cite{HL2} studied 0-1 matrices whose powers have  bounded entries. Chen, Huang and Yan \cite{CHY} studied the {\it stable index} of 0-1 matrices, which is defined based on the properties of the powers of 0-1 matrices.

Let $A$ be a square 0-1 matrix and let $k$ be a positive integer. If $A^2=A$, then $A$ is said to be {\it idempotent}; if $A^k=A$, then $A$ is said to be {\it $k$-idempotent}. Note that an idempotent matrix is $k$-idempotent for every positive integer $k\ge 2$. Ma \cite{CM} characterized the idempotent 0-1 matrices.  We solve the following problem in this paper.

\begin{problem}
Let $n$ and $k$ be positive integers. Characterize the $k$-idempotent 0-1 matrices of order $n$.
Determine the maximum number of nonzero entries in $k$-idempotent 0-1 matrices of order $n$ as well as the $k$-idempotent 0-1 matrices attaining the maximum number.
\end{problem}

\section{Structures of $k$-idempotent 0-1 matrices}
\setcounter{theorem}{0}

We need some terminology on digraphs.
Let $D$ be a digraph. We denote by $\V(D)$ the vertex set and   $\A(D)$ the arc set of $D$. The arc from  vertex $u$ to vertex $v$ is denoted by  $(u,v)$ or $uv$. If $(u,v)\in \A(D)$, then $u$ is  an {\it in-neighbour} of $v$ and $v$ is an {\it out-neighbour} of $u$. A   {\it directed walk}  is sequence of
consecutive arcs $(v_1,v_2),(v_2,v_3),\ldots,(v_{t-1},v_t)$, which is  written as $v_1v_2\cdots v_t$. A {\it directed cycle}   is a closed walk $v_1v_2\cdots v_tv_1$ with $v_1,v_2\ldots,v_t$ being distinct.  A {\it directed path}   is a directed walk in which all the vertices are distinct. Directed walk, directed path and directed cycle will be abbreviated as walk, path and cycle in this paper. The number of arcs in a  walk (cycle) is called its {\it length}. A  walk (cycle) of length $k$ is called a {\it $k$-walk ($k$-cycle)}. If a digraph $D$ contins a walk from $u$ to $v$ for all  $u,v\in\V(D)$, then $D$ is said to be {\it strongly connected}.

Given a digraph $D=(\mathcal{V},\mathcal{A})$ with vertex set $\mathcal{V}=\{v_1,v_2,\ldots,v_n\}$ and arc set $\mathcal{A}$, its {\it adjacency matrix} $A_D=(a_{ij})$ is defined by
\begin{equation*}
a_{ij}=\left\{\begin{array}{ll}
1,&\textrm{if } (v_i,v_j)\in \mathcal{A};\\
0,&\textrm{otherwise}.\end{array}\right.
\end{equation*}

 The set of 0-1 matrices of order $n$ is denoted by  $M_n\{0,1\}$.
 Let $A=(a_{ij})\in M_n\{0,1\}$.  We   define its digraph as $D(A)=(\mathcal{V},\mathcal{A})$  with $\mathcal{V}=\{1,2,\ldots,n\}$ and  $\mathcal{A}=\{(i,j): a_{ij}=1, 1\le i,j\le n\}$.

 Denote by $A(i,j)$ the $(i,j)$-entry of $A$. Then for $A\in M_n\{0,1\}$, $A^k(i,j)=t$ if and only if $D(A)$ has exactly $t$ distinct walks of length $k$ from $i$ to $j$.
 Therefore, if $A^k=A$, then $D(A)$ has an arc $(i,j)$ if and only if it has a $k$-walk from $i$ to $j$.

   A nonnegative square matrix $A$ is called {\it reducible}, if it is permutation similar to a matrix of the form
  $$\left(\begin{array}{cc}
B&C\\0&D
\end{array}
\right), $$ where $B$ and $D$ are square matrices. If $A$ is not reducible, then it is called {\it irreducible}.  For convenience, a 1-by-1 matrix is always called irreducible. A square matrix A is irreducible if and only if its digraph $D(A)$ is strongly connected (see \cite{XZ}, p.131).

Denote by $r_i(A)$ the $i$-th row sum and $\rho(A)$ the spectral radius  of a matrix $A$. We have the following.
\begin{lemma}{\rm\cite{Minc}}\label{Minc} Let $A$ be an $n\times n$ nonnegative matrix. Then
$$\min_{1\leq i\leq n}r_i(A)\leq \rho(A)\leq \max_{1\leq i\leq n}r_i(A).$$
Moreover, if $A$ is irreducible, then one of equalities holds if and
only if $r_1(A)=r_2(A)=\cdots=r_n(A).$
\end{lemma}

Denote by $\overrightarrow{C}_k$ the directed cycle with $k$ vertices and $C_k$ the adjacency matrix of  $\overrightarrow{C}_k$, which is the $k\times k$ basic circulant matrix
 $$\left(\begin{array}{ccccc}
  &1& & \\
  & &\ddots& \\
    &&&1\\
 1&&&\end{array}
 \right).$$
From the above lemma, we deduce the following result immediately.
\begin{lemma}\label{le1}
Let $A$ be an irreducible 0-1 matrix of order $n$. If $\rho(A)=1$, then $A$ is permutation similar to $C_n$; if $\rho(A)=0$, then $n=1$ and $A=0$.
\end{lemma}
\begin{lemma}\label{le2}
Let $k\ge 2$ be an integer, and let $A$ be a strictly upper triangular 0-1 matrix of order $n\ge 2$. If $A^k=A$, then $A=0$.
\begin{proof}
To the contrary, suppose $A\ne 0$. Then $D(A)$ contains an arc $(u_1,v_1)$, which implies that $D(A)$ contains a $k$-walk $W_1\equiv u_1u_2u_3\cdots u_{k}v_1$ from $u_1$ to $v_1$. Since $A$ is strictly upper triangular, $D(A)$ is acyclic. Therefore, the vertices $u_1,u_2,u_3,\ldots,u_k,v_1$ are distinct and $W_1$ is a path. Now $D(A)$ containing the arc $(u_1,u_2)$ implies that it contains another $k$-path $W_2$ from $u_1$ to $u_2$.  Combining $W_1$ and $W_2$ we get a $(2k-1)$-path from $u_1$ to $v_1$. Now if $2k-1>n$, then we get a contradiction since $D(A)$ has only $n$ vertices. If $2k-1\le n$,   repeating the above arguments we can always find a path with length larger than $n$ in $D(A)$, a contradiction.
\end{proof}
\end{lemma}

 Now we are ready to present the characterization of $k$-idempotent 0-1 matrices.
\begin{theorem}\label{th3}
A square 0-1 matrix $A$ is $k$-idempotent if and only if $A=0$ or $A$ is  permutation similar to
$$\left(\begin{array}{ccc}
0&X&XP^TY\\0&P&Y\\0&0&0
\end{array}
\right), $$
where the diagonal zero blocks are square and may vanish, $P=\oplus_{i=1}^r C_{n_i}$ with  $n_i \mid k-1$ for $i=1,\ldots,r$, $X$ and $Y$ are 0-1 matrices such that $XP^TY$ is also a 0-1 matrix.
\end{theorem}
\begin{proof}
Let $A\in M_n\{0,1\}$ be $k$-idempotent. Then $A$ is permutation similar to an upper triangular block matrix
$$B=\left(\begin{array}{cccc}
A_1&&* \\&\ddots&\\0&&A_s
\end{array}
\right), $$
where each $A_i$ is an $n_i\times n_i$ irreducible matrix for $i=1,\ldots,s$.

Since $A^k=A$, we have
$$A_i^k=A_i\quad {\rm for}\quad i=1,\ldots,s.$$
It follows that $$\rho^k(A_i)=\rho(A_i^k)=\rho(A_i), $$which implies
$$\rho(A_i)=1~ {\rm or} ~0\quad {\rm for}\quad i=1,\ldots,s.$$
Applying Lemma \ref{le1} on each $i\in \{1,\ldots,s\}$, if $\rho(A_i)=0$, then $n_i=1$ and $A_i=0$;
if $\rho(A_i)=1$, then $A_i$ is permutation similar to $C_{n_i}$ with $n_i\mid k-1$, since it easy to check that $C_{n_i}^k\ne C_{n_i}$ when $n_i$ does not divide $k-1$.

Therefore, without loss of generality, we may assume $A_i=C_{n_i}$ if $A_i\ne 0$.
Now we prove the following claim.\\

{\bf Claim 1.} {\it Suppose $A_i=C_{n_i}, A_j=C_{n_j}$ with $i,j\in \{1,\ldots,s\}$. Then $D(A)$ contains no path from the vertices of $D(A_i)$ to the vertices of $D(A_j)$.}

To the contrary, suppose $D(A)$ contains a path  $u_1u_2\cdots u_t$  with $u_1\in \V(D(A_i))$ and $v_t\in\V(D(A_j))$.

Firstly, we prove that there is an arc from a vertex of $D(A_i)$ to a vertex of $D(A_j)$.
Note that $(u_{t-1},u_t)\in \A(D)$ implies that there is a $k$-walk from $u_{t-1}$ to $u_t$, which consists of the arc $(u_{t-1},u_t)$ and $(k-1)/n_j$ copies of the cycle $D(A_j)$. If $t\ge 3$, then by using the cycle $D(A_j)$ repeatedly, we can find a $k$-walk from $u_{t-2}$ to a vertex $v_1\in \V(D(A_j))$, which is the in-neighbour of $u_t$ in $D(A_j)$. Since $A^k=A$, we have $(u_{t-2},v_1)\in \A(D(A))$. If $t-2=1$, then  $(u_{t-2},v_1)$ is the arc we need. If $t-2>1$, repeating the above arguments we can always find an arc from $u_1$ to some vertex $v$ of $D(A_j)$.

Now assume $D(A)$ contains an arc $(u,v)$ with $u\in\V(D(A_i))$ and $v\in \V(D(A_j))$. Then we can find two $k$-walks from $u$ to $v$ in $D(A)$: one walk consisting  of $(k-1)/n_i$ copies of $D(A_i)$ and the arc $(u,v)$, and another walk consisting of the arc $(u,v)$ and $(k-1)/n_j$ copies of $D(A_j)$. Hence, we get $A^k\ne A$, a contradiction. This completes the proof of Claim 1.\\

Applying Claim 1, if $B$ is not the zero matrix, then it is permutation similar to
$$H=\left(\begin{array}{cccc}
0&X&Z\\0&P&Y\\0&0&0
\end{array}
\right), $$
where $P=\oplus_{i=1}^r C_{n_i}$ with  $n_i \mid k-1$ for $i=1,\ldots,r$, the diagonal zero block matrices are square, which may vanish.
By a direct computation, we have
 $$H^m=\left(\begin{array}{cccc}
0&XP^{m-1}&XP^{m-2}Y\\0&P^m&P^{m-1}Y\\0&0&0
\end{array}
\right).$$
On the other hand,  $P^{k}=P$ implies  $P^{k-1}=I$ and $P^{k-2}=P^T$. Therefore,  $$H^k=\left(\begin{array}{cccc}
0&X&XP^TY\\0&P&Y\\0&0&0
\end{array}\right).$$Now $A^k=A$ implies $H^k=H$, and hence $Z=XP^TY$. This completes the proof.
\end{proof}

When $k=2$, each $C_{n_i}$ in Theorem \ref{th3} is the 1-by-1 matrix 1. Hence, Ma's characterization of idempotent 0-1 matrix follows from Theorem \ref{th3} directly.
\begin{corollary}\cite{CM}
A square 0-1 matrix is  idempotent if and only if it is permutation similar to
$$\left(\begin{array}{ccc}
0&X&XY\\0&I&Y\\0&0&0
\end{array}
\right), $$
where the zero diagonal blocks are square and  may vanish.
\end{corollary}

\section{Maximum number of nonzero entries in $k$-idempotent 0-1 matrices}

 In this section, we determine the maximum number of nonzero entries in $k$-idempotent 0-1 matrices of a given order as well as the matrices attaining this maximum number.
Denote by $f(A)$ the number of nonzero entries in a 0-1 matrix $A$, $O_n$ the zero square matrix of order $n$,  and $J_{m,n}$ the $m\times n$ matrix with all entries equal to 1.
Given a positive integer $n$, we define
\begin{equation*}
\gamma(n)\equiv \begin{cases}
\frac{(n+1)^2}{4},&\text{if $n$ is odd,}\\
\frac{n^2+2n}{4},&\text{if $n$ is even}.
\end{cases}
\end{equation*} We need the following lemmas.
\begin{lemma}\label{le4}
Let $n\ge 2$ be an integer. Suppose $A\in M_n\{0,1\}$ has the form
$$\left(\begin{array}{cc}
O_r&X\\0&P
\end{array}\right)\quad {\rm or }\quad \left(\begin{array}{cc}
P&X\\0&O_r
\end{array}\right) $$
with $P$ being a permutation matrix. Then
$$
f(A)\le \gamma(n)
$$
  with equality   if and only if $X$ is a matrix with all entries equal to 1,  $r=(n-1)/2$ when $n$ is odd and $r\in\{n/2,n/2-1\}$ when $n$ is even.
\end{lemma}
\begin{proof}
Note that
\begin{eqnarray*}
f(A)=f(P_s)+f(X)
\le (n-r)+r(n-r)=(r+1)(n-r).
\end{eqnarray*}  The result follows directly.
\end{proof}
Given a block matrix $\left(\begin{array}{ccc}
O_r&A\\B&C
\end{array}\right)$, if $r=0$, then both the first  row  and the first column vanish, i.e., the matrix is $C$.
\begin{lemma}
Let $n\ge 3$ be an integer. Suppose $A\in M_n\{0,1\}$ has the form
\begin{equation}\label{eqh1}
\left(\begin{array}{ccc}
O_r&X&Z\\0&P&Y\\0&0&O_s
\end{array}\right)
\end{equation}
  with $P$ being a permutation matrix, $r\ge 0$, $s\ge 0$. If $A^2$ is  a 0-1 matrix, then
$$f(A)\le \gamma(n)$$
with equality if and only if one of the following holds.
\begin{itemize}
    \item[(a)] $r=(n-1)/2$ when $n$ is odd and $r\in\{n/2,n/2-1\}$ when $n$ is even; $X$ and $Z$ are matrices with all entries equal to 1; each column of $Y$ has exactly one nonzero entry.
    \item[(b)] $s=(n-1)/2$ when $n$ is odd and $s\in\{n/2,n/2-1\}$ when $n$ is even; $Y$ and $Z$ are matrices with all entries equal to 1; each row of $X$ has exactly one nonzero entry.
\end{itemize}
\end{lemma}
\begin{proof}
We apply a similar strategy as in the proof of [10, Theorem 2]. We use induction on the order $n$. If $r=0$ or $s=0$, then the result follows from Lemma \ref{le4}. Now suppose $r\ge 1$ and $s\ge 1$.

If $n=3$, then
$$A=\left(\begin{array}{ccc}
0 &x&z\\
0&1&y\\0&0&0
\end{array}\right)$$
and $f(A)$ attains the maximum if and only if $x=y=z=1$. Hence the result holds for $n=3$.

If $n=4$, then we have three possibilities: $r=1,s=2,P=1$; $r=2,s=1,P=1$; $r=s=1$, $P\in M_2\{0,1\}$, which means $A$ has the following forms
$$\left(\begin{array}{cccc}
0 &x&y&z\\
0&1&u&v\\0&0&0&0\\0&0&0&0
\end{array}\right),\quad
\left(\begin{array}{cccc}
0 &0&u&x\\
0&0&v&y\\0&0&1&z\\0&0&0&0
\end{array}\right),\quad\left(\begin{array}{ccc}
0 &\alpha^T&x\\
0&P&\beta\\0&0&0
\end{array}\right). $$
In the first two cases, it is obvious that $f(A)$ attains the maximum if and only if $x=y=z=u=v=1$. In the last case, note that $A^2$ being a 0-1 matrix implies
$f(\alpha)+f(\beta)\le 3.$
 Thus $f(A)\le 6$ with equality if and only if $x=1$ and $
f(\alpha)+f(\beta)= 3.$ Therefore, the result holds for $n=4$.

Now we assume $n\ge 5$ and the result holds for  0-1 matrices of order less than $n$. Suppose $A\in M_n\{0,1\}$ has form (1) and $A^2$ is a 0-1 matrix. Partition $A$ as
$$A=\left(\begin{array}{ccccc}
0&0&u_1^T&u_2^T&z\\
0&O_{r-1}&X_1&Z_1&v_1\\
0&0&P&Y_1&v_2\\
0&0&0&O_{s-1}&0\\
0&0&0&0&0
\end{array}\right)\equiv\left(\begin{array}{cccc}
0&u^T&z\\
0&B&v\\0&0&0
\end{array}\right)$$
where $u$ and $v$ are 0-1 column vectors.
Note that $A^2(1,n)\le 1$ implies
\begin{equation}\label{eqh2}
f(u)+f(v)+z\le n
\end{equation}
 with equality  if and only if
 \begin{equation}\label{eqh3}
 z=1, v_1=J_{r-1,1}, u_2=J_{s-1,1}\quad f(u_1)+f(v_2)=n-r-s+1,\quad u^Tv=u_1^Tv_2=1.
 \end{equation}
We have
\begin{eqnarray*}
f(A)=f(B)+f(u)+f(v)+z
     \le\gamma(n-2)+n=\gamma(n).
\end{eqnarray*}
Moreover, $f(A)=\gamma(n)$ if and only if $f(B)=\gamma(n-2)$ and equality in (\ref{eqh2}) holds. By the induction hypothesis,
we have (\ref{eqh3}) and one of the following holds.
\begin{itemize}
    \item[(i)] $r-1=(n-3)/2$ when $n$ is odd and $r-1\in\{(n-2)/2,(n-2)/2-1\}$ when $n$ is even; $X_1$ and $Z_1$ are matrices with all entries equal to 1; each column of $Y_1$ has exactly one nonzero entry.
    \item[(ii)] $s-1=(n-3)/2$ when $n$ is odd and $s-1\in\{(n-2)/2,(n-2)/2-1\}$ when $n$ is even; $Y_1$ and $Z_1$ are matrices with all entries equal to 1; each row of $X_1$ has exactly one nonzero entry.
\end{itemize}
If (i) holds, then
$$A=\left(\begin{array}{ccccc}
0&0&u_1^T&J_{1,s-1}&1\\
0&O_{r-1}&J_{r-1,t}&J_{r-1,s-1}&J_{r-1,1}\\
0&0&P&Y_1&v_2\\
0&0&0&O_{s-1}&0\\
0&0&0&0&0
\end{array}\right) $$
where $t=n-r-s$ and $f(u_1)+f(v_2)=t+1$. Considering $A^2(2,n)\le 1$ we have $f(v_2)=1$ and $u_1=J_{1,t}$. Therefore, $A$ satisfies statement (a).

If (ii) holds, then
$$A=\left(\begin{array}{ccccc}
0&0&u_1^T&J_{1,s-1}&1\\
0&O_{r-1}&X_1&J_{r-1,s-1}&J_{r-1,1}\\
0&0&P&J_{t,s-1}&v_2\\
0&0&0&O_{s-1}&0\\
0&0&0&0&0
\end{array}\right) $$
where $t=n-r-s$ and $f(u_1)+f(v_2)=t+1$. Considering $A^2(1,n-1)\le 1$ we have $f(u_1)=1$ and $v_2=J_{t,1}$. Therefore, $A$ satisfies statement (b).

On the other hand, if $A$ has form (\ref{eqh1}) with (a) or (b) holding, then it is easy to check that $A^2\in M_n\{0,1\}$ and $f(A)=\gamma(n)$.
This completes the proof.
\end{proof}
\begin{theorem}
Let $A\in M_n\{0,1\}$ be  $k$-idempotent. Then
$$f(A)\le\gamma(n)$$
    with equality if and only if $A$ is permutation similar to a matrix with form (\ref{eqh1}) such that statement (a) or (b) in Lemma \ref{le4} holds, and $P=\oplus_{i=1}^t C_{n_i}$ with  $n_i \mid k-1$ for $i=1,\ldots,t$.
\end{theorem}
\begin{proof}
By Theorem \ref{th3}, $A$ is permutation to a matrix with form (\ref{eqh1}) such that $Z=XP^TY$, and $P=\oplus_{i=1}^t C_{n_i}$ with  $n_i \mid k-1$ for $i=1,\ldots,t$. Consider the 0-1 matrix
$$B=\left(\begin{array}{ccc}
O_r&X&Z\\0&P&P^TY\\0&0&O_s
\end{array}\right).$$
We have $$B^2=\left(\begin{array}{ccc}
O_r&X&Z\\0&P^2&P^TY\\0&0&O_s
\end{array}\right)\in M_n\{0,1\}.$$
Applying Lemma \ref{le4}, we get
$$f(A)=f(B)\le \gamma(n)$$
with equality if and only if one of the following holds.
\begin{itemize}
    \item[(i)] $r=(n-1)/2$ when $n$ is odd and $r\in\{n/2,n/2-1\}$ when $n$ is even; $X$ and $Z$ are matrices with all entries equal to 1; each column of $P^TY$ has exactly one nonzero entry.
    \item[(ii)] $s=(n-1)/2$ when $n$ is odd and $s\in\{n/2,n/2-1\}$ when $n$ is even; $P^TY$ and $Z$ are matrices with all entries equal to 1; each row of $X$ has exactly one nonzero entry.
\end{itemize}
This completes the proof.
\end{proof}
\section*{Acknowledgement}

The first author was supported by a Natural Science Fund of Shenzhen University.
The second author was supported by National Natural Science Foundation of China (No. 11771141)

\end{document}